\documentclass[12pt]{amsart}
\usepackage{a4wide,enumerate,xcolor}
\usepackage{amsmath}
\allowdisplaybreaks
\usepackage{color}

\let\pa\partial
\let\na\nabla
\let\eps\varepsilon
\newcommand{\N}{{\mathbb N}}
\newcommand{\R}{{\mathbb R}}
\newcommand{\diver}{\operatorname{div}}

\newtheorem{theorem}{Theorem}
\newtheorem{lemma}[theorem]{Lemma}

\newtheorem{remark}[theorem]{Remark}

\newtheorem{definition}{Definition}

\begin{document}

\title[Chemotaxis compressible Navier--Stokes equations]{Global existence and weak-strong uniqueness for chemotaxis compressible 
Navier--Stokes equations modeling vascular network formation}

\author[X. Huo]{Xiaokai Huo}
\address{Department of Mathematics, Iowa State University, 
411 Morrill Road, Ames IA 50011-2104, USA}
\email{xhuo@iastate.edu}

\author[A. J\"ungel]{Ansgar J\"ungel}
\address{Institute of Analysis and Scientific Computing, Technische Universit\"at Wien,
Wiedner Hauptstra\ss e 8--10, 1040 Wien, Austria}
\email{juengel@tuwien.ac.at}

\date{\today}

\thanks{The last author acknowledges partial support from
the Austrian Science Fund (FWF), grants P33010 and F65.
This work has received funding from the European
Research Council (ERC) under the European Union's Horizon 2020 research and
innovation programme, ERC Advanced Grant no.~101018153.}

\begin{abstract}
A model of vascular network formation is analyzed in a bounded domain, consisting of
the compressible Navier--Stokes equations for the density of the endothelial
cells and their velocity, coupled to a reaction-diffusion equation for the 
concentration of the chemoattractant, which triggers the migration of the endothelial
cells and the blood vessel formation. The coupling of the equations is realized by
the chemotaxis force in the momentum balance equation. The global existence of finite energy
weak solutions is shown for adiabatic pressure coefficients $\gamma>8/5$.
The solutions satisfy a relative energy inequality, which allows for the proof of
the weak--strong uniqueness property.
\end{abstract}



\maketitle


\section{Introduction}

The formation of blood vessels is regulated by chemical signals triggering the movement of
endothelial cells. The cells may self-assemble into a vascular network, which is known as 
vasculogenesis. In this paper, we analyze a mathematical model for the formation of
vascular networks, based on mass and momentum balance equations including a chemotaxis force
and coupled with a reaction-diffusion equation for the signal concentration.
The existence of global weak solutions to the resulting chemotaxis compressible
Navier--Stokes equations was proved in \cite{AiAl16} for pressures with adiabatic
exponent $\gamma>3$. We extend the existence result to the range $\gamma > 8/5$ and
prove a weak--strong uniqueness property. The proofs are based on
a new relative energy inequality.

The dynamics of the density $\rho(x,t)$ of the endothelial cells, their velocity $v(x,t)$,
and the concentration $c(x,t)$ of the chemoattractant  
(e.g.\ the vascular endothelial growth factor VEGF-A \cite{SAGGPB03})
is given by the equations
\begin{align}
  & \pa_t\rho + \diver(\rho v) = 0, \label{1.mass} \\
	& \pa_t(\rho v) + \diver(\rho v\otimes v) + \na p(\rho) 
	= \mu\Delta v + (\lambda+\mu)\na\diver v + \rho\na c - \frac{\rho v}{\zeta}, 
	\label{1.mom} \\
	& \pa_t c = \Delta c - c + \rho\quad\mbox{in }\Omega,\ t>0, \label{1.chem}
\end{align}
where $\Omega\subset\R^3$ is a bounded domain, $p(\rho)=\rho^\gamma$ with the adiabatic
exponent $\gamma>1$ is the pressure, 
the Lam\'e viscosity constants $\mu$, $\lambda$ satisfy $\mu>0$ and
$3\lambda + 2\mu>0$, and $\zeta>0$ is a relaxation constant. 
We impose the initial and boundary conditions
\begin{align}
  & \rho(\cdot,0) = \rho^0, \quad v(\cdot,0) = v^0, \quad c(\cdot,0) = c^0
	\quad\mbox{in }\Omega, \label{1.ic} \\
	& v = 0, \quad \na c\cdot\nu = 0 \quad\mbox{on }\pa\Omega,\ t>0. \label{1.bc}
\end{align}
The boundary condition for the velocity $v$ is the no-slip condition, and the
no-flux boundary condition for $c$ means that there is no inflow or outflow
of the concentration.  
The momentum balance equation \eqref{1.mom} includes viscous terms as in \cite{AiAl16}
(suggested in \cite[p.~1862]{AGS04}) 
as well as the chemotaxis force $\rho f_{\rm chem}=-\rho\na c$
and the drag force $\rho f_{\rm drag}=-\rho v/\zeta$. The reaction-diffusion equation
\eqref{1.chem} for the signal concentration models diffusion in the surrounding medium,
degradation of the signal in finite time, and the release of the signal produced by the
cells. We have set the physical constants in 
\eqref{1.mass}--\eqref{1.chem} equal to one, except $\zeta$ to distinguish terms originating
from the drag force.

The existence of global finite energy weak solutions to \eqref{1.mass}--\eqref{1.bc} has been
proved in \cite{AiAl16} for $\gamma>3$. This restriction comes from the estimation of the chemotaxis force; see Remark \ref{rem.AiAl16} on page \pageref{rem.AiAl16}. We extend the existence result to $\gamma>8/5$ by rewriting the force term $\rho\na c$ via \eqref{1.chem} as $(\pa_t c-\Delta c+c)\na c$ and exploiting the properties of the Bogovskii operator. Replacing the parabolic equation \eqref{1.chem} for $c$ by the elliptic one, we can even allow for $\gamma > 3/2$, which is the
condition needed in the existence theory of the compressible Navier--Stokes equations \cite{FNP01}; see Remark \ref{rem.improve}. This may indicate that our condition $\gamma>8/5$ for system \eqref{1.mass}--\eqref{1.chem} is not optimal. We discuss this issue further in Remark \ref{rem.gamma}.

The idea of the existence proof in \cite{AiAl16} is to derive a priori estimate via the 
energy-type functional
$$
  \widetilde{H}(\rho,v,c) = \int_\Omega\bigg(\psi(\rho) + \frac12\rho|v|^2 + \frac12 c^2\bigg)dx,
$$
where $\psi(\rho)=\rho\int_0^\rho s^{-2}p(s)ds=\rho^\gamma/(\gamma-1)$ can be interpreted
as the internal energy. Unfortunately, this functional is not
bounded as $t\to\infty$. Our idea is to use the physical (free) energy,
\begin{equation}\label{1.E}
  E(\rho,v,c) = \int_\Omega\bigg(\psi(\rho) + \frac12\rho|v|^2 + \frac12(|\na c|^2+c^2)
	- \rho c\bigg)dx,
\end{equation}
which is the sum of the kinetic energy $\frac12\int_\Omega\rho|v|^2dx$ 
and the energy $E(\rho,0,c)$ of the parabolic--parabolic Keller--Segel model.
We show in Section \ref{sec.ex} (see Lemma \ref{lem.Hn} on page \pageref{lem.Hn}) that 
\begin{align*}
  \frac{dE}{dt}(\rho,v,c) + \int_\Omega\big(\mu|\na v|^2 + (\lambda+\mu)|\diver v|^2\big)dx
	+ \int_\Omega|\pa_t c|^2 dx \le 0,
\end{align*}
providing a bound for $E((\rho,v,c)(t))$ uniformly in time.
Clearly, to infer a priori estimates, we need an upper bound for $\rho c$. This is done by using the inequality 
\begin{equation}\label{1.sugiyama}
  \int_\Omega\rho cdx \le \frac{1}{2}\|\psi(\rho)\|_{L^1(\Omega)}
	+ \frac14\|\na c\|_{L^2(\Omega)}^2 + C_1(\gamma)\|c\|_{L^1(\Omega)}^{C_2(\gamma)},
\end{equation}
which is due to Sugiyama \cite{Sug07} (see Lemma \ref{lem.aux} on page \pageref{lem.aux}), where $C_1(\gamma)>0$ and $C_2(\gamma)>0$ only depend on $\gamma$, and which \eqref{1.sugiyama} requires the condition $\gamma>8/5$. The first two terms on the right-hand side of \eqref{1.sugiyama} can be absorbed by the energy, while the $L^1(\Omega)$ norm of $c$ 
can be bounded in terms of the initial data $(\rho^0,c^0)$. This provides a bound
for the modified energy-type functional
\begin{equation}\label{1.H}
  H(\rho,v,c) = \frac12\int_\Omega\bigg(\psi(\rho) + \rho|v|^2 + \frac12|\na c|^2 + c^2\bigg)dx,
\end{equation}
namely
\begin{align*}
  H((\rho,v,c)(t)) &+ \int_0^t\int_\Omega\big(\mu|\na v|^2+(\lambda+\mu)|\diver v|^2\big)dxds \\
	&{}+ \int_0^t\int_\Omega|\pa_s c|^2 dxds \le C(\rho^0,v^0,c^0),
\end{align*}
which allows us to prove the global existence of finite energy weak solutions such that
$H(\rho,v,c)$ is finite for all $t>0$. This type of solutions is defined as follows.

\begin{definition}[Finite energy weak solution]\label{def.weak}
The triple $(\rho,v,c)$ is a {\em finite energy weak solution} 
to \eqref{1.mass}--\eqref{1.bc} if 
\begin{itemize}
\item they satisfy the regularity
\begin{align*}
  & \rho\in L^\infty(0,T;L^\gamma(\Omega)), \quad \rho\ge 0\mbox{ in }\Omega,\ t>0, \\
	& v\in L^2(0,T;H_0^1(\Omega;\R^3)), \quad 
	c\in L^\infty(0,T;H^1(\Omega))\cap H^1(0,T;L^2(\Omega));
\end{align*}
\item equation \eqref{1.mass} is satisfied in the sense of renormalized solutions
\cite[Section 10.18]{FeNo09};
\item equations \eqref{1.mom}--\eqref{1.chem} are satisfied in the sense of distributions;
\item the energy inequality
\begin{align*}
  E((&\rho,v,c)(t)) 
	+ \int_0^t\int_\Omega\big(\mu|\na v|^2 + (\lambda+\mu)|\diver v|^2\big)dxds \\
	&{}+ \int_0^t\int_\Omega|\pa_s c_n|^2 dxds
	+ \frac{1}{\zeta}\int_0^t\int_\Omega \rho_n|v_n|^2 dxds
	\le E(\rho^0,v^0,c^0)
\end{align*}
holds for a.e.\ $t\in(0,T)$. 
\end{itemize}
\end{definition}
We introduce for $1<p,q<\infty$ the space $W^{2-2/p,q}_\nu( \Omega)$ as the completion of the space of functions $w\in C^\infty(\overline\Omega)$ satisfying $\na w\cdot\nu=0$ on $\pa\Omega$ in the norm of $W^{2-2/p,q}(\Omega)$. We can now state our first main result.

\begin{theorem}[Global existence]\label{thm.ex}
Let $\pa\Omega\in C^2$, $p(\rho)=\rho^\gamma$ for $\rho\ge 0$ with $\gamma>8/5$.
Assume that the initial datum satisfies $\rho^0\in L^\gamma(\Omega)$ with
$\rho^0\ge 0$, $\rho^0\not\equiv 0$ in $\Omega$, $\rho^0|v^0|^2\in L^1(\Omega)$,
and $c^0\in W^{2-2/\gamma,\gamma}_\nu(\Omega)$, $c^0\ge 0$ in $\Omega$. 
Then there exists a finite energy weak solution
$(\rho,v,c)$ to \eqref{1.mass}--\eqref{1.bc} in the sense of Definition \ref{def.weak}.
\end{theorem}

The condition on the initial datum  $c^0\in W^{2-2/\gamma,\gamma}_\nu(\Omega)$ can be rephrased in terms of interpolation or Besov spaces. Indeed, the condition is needed to apply the maximal regularity result of Theorem \ref{thm.maxreg}, and the regularity on the initial datum can be formulated in such spaces; see \cite[Theorem 10.22]{FeNo09}.
The definition of the pressure can be relaxed to 
$p\in C^0([0,\infty))\cap C^2(0,\infty)$, $p(0)=0$, $p'(\rho)>0$ for $\rho>0$, and
$\rho^{1-\gamma}p'(\rho)\to a>0$ as $\rho\to\infty$; see \cite[(2.1)]{FJN12}. 
The proof of the theorem is based on the existence theory for the compressible
Navier--Stokes equations \cite{FNP01}. More precisely, we add some artificial diffusion
and an artificial pressure term, construct Faedo--Galerkin solutions to the 
approximate problem, prove an approximate energy inequality for these solutions, 
and pass to the de-regularizing limit. Improved uniform bounds for the cell density in
$L^{\gamma+\theta}(\Omega)$ for some $\theta>0$ are derived by testing the mass balance
equation with a test function involving the Bogovskii operator. The novel part
is the estimate of the chemotaxis force term.

Next, we formulate the weak--strong uniqueness property of the system, meaning that a weak and 
a strong solution emanating from the same initial data coincide as long as the latter
exists. 

\begin{theorem}[Weak--strong uniqueness]\label{thm.wsu}
Let $(\rho,v,c)$ and $(\bar\rho,\bar v,\bar c)$ be two finite energy weak solutions
to \eqref{1.mass}--\eqref{1.bc} constructed in Theorem \ref{thm.ex} with the
same initial data. Assume that $(\bar\rho,\bar v,\bar c)$ satisfies the additional
regularity
\begin{align}\label{1.regul}
  0<c_p\le\bar\rho\le C_p, \ |\bar v|\le C_v\mbox{ a.e. in }\Omega\times(0,T), \quad
	|\na\bar\rho|,\,|\na^2\bar v|\in L^2(0,T;L^q(\Omega))
\end{align}
for $q>3$ and some constants $c_p,C_p,C_v>0$.
Then $\rho=\bar\rho$, $v=\bar v$, and $c=\bar c$ in $\Omega\times(0,T)$.
\end{theorem}

The no-vacuum assumption $\bar\rho\ge c_p>0$ was also needed in \cite{FNS11} and in related contexts, e.g.\ for the weak--strong uniqueness property of Maxwell--Stefan systems \cite{HJT22}.
The proof of Theorem \ref{thm.wsu} is based on the relative energy method.
The relative energy, associated to the energy functional \eqref{1.E}, is given by
\begin{equation*}
  E(\rho,v,c|r,u,z) = \int_\Omega\bigg(\psi(\rho|r) + \frac12\rho|v-u|^2
	+ \frac12\big(|\na(c-z)|^2 + (c-z)^2\big) - (\rho-r)(c-z)\bigg)dx,
\end{equation*}
where $\psi(\rho|r) = \psi(\rho)-\psi(r)-\psi'(r)(\rho-r)$ is the 
Bregman distance associated to $\psi$.
We show in Lemma \ref{lem.relE} on page \pageref{lem.relE} that
\begin{align}\label{1.relEineq}
  E((&\rho,v,c)(t)|(r,u,z)(t)) + \int_0^t\int_\Omega\big(\mu|\na(v-u|^2
	+ (\lambda+\mu)|\diver(v-u)|^2\big)dxds \\
	&{}+ \int_0^t\int_\Omega|\pa_s(c-z)|^2 dxds
	\le E(\rho^0,v^0,c^0|r^0,u^0,z^0) + \int_0^t R(\rho,v,c|r,u,z)ds, \nonumber
\end{align}
where $(\rho,v,c)$ is a finite energy weak solution to \eqref{1.mass}--\eqref{1.bc},
$(r,u,z)$ are smooth functions, and the remainder
$R(\rho,v,c|r,u,z)$ is defined in Lemma \ref{lem.relE} below. 
Finite energy weak solutions
to the compressible Navier--Stokes equations satisfying \eqref{1.relEineq} have been
called {\em suitable weak solutions} in \cite{FNS11}. It was shown in \cite{FJN12} 
that finite energy weak solutions in fact always satisfy the relative energy inequality
\eqref{1.relEineq} for smooth functions $(r,u,z)$.

Defining the modified relative energy 
$$
  H(\rho,v,c|r,u,z) = \frac12\int_\Omega\bigg(\psi(\rho|r) + \rho|v-u|^2
	+ \frac12|\na(c-z)|^2 + (c-z)^2\bigg)dx
$$	
and giving another weak solution $(r,u,z)=(\bar\rho,\bar{v},\bar{c})$ 
satisfying the regularity \eqref{1.regul}, the idea of the proof is to show that
\begin{align*}
  & R(\rho,v,c|\bar\rho,\bar{v},\bar{c}) \le CH(\rho,v,c|\bar\rho,\bar{v},\bar{c})
	\quad\mbox{and} \\
	& \int_\Omega(\rho-\bar\rho)(c-\bar{c})dx \le \frac12 H(\rho,v,c|\bar\rho,\bar{v},\bar{c})
	+ C(\rho^0,c^0),
\end{align*}
which leads to
$$
  \frac12 H((\rho,v,c)(t)|(\bar\rho,\bar{v},\bar{c})(t)) \le E(\rho^0,v^0,c^0|r^0,u^0,z^0)
	+ C(\rho^0,c^0) + C\int_0^t H(\rho,v,c|\bar\rho,\bar{v},\bar{c})ds
$$
and which implies, by Gronwall's lemma, that $H((\rho,v,c)(t)|(\bar\rho,\bar{v},\bar{c})(t))=0$,
Consequently $\rho(t)=\bar\rho(t)$, $v(t)=\bar v(t)$, and $c(t)=\bar c(t)$ for $t>0$.

We finish the introduction by discussing the state of the art.
The global existence of finite energy weak solutions to the compressible Navier--Stokes
equations with adiabatic exponents $\gamma>3/2$ 
was shown in \cite{FNP01}. 
The range of $\gamma$ can be extended to $\gamma>1$ for axisymmetric initial data \cite{JiZh03}
or for a class of density-dependent viscosity coefficients \cite{MeVa07}, for instance.
Germain \cite{Ger11} proved a relative energy
inequality and established the weak--strong uniqueness property for solutions 
to the compressible Navier--Stokes equations with an 
integrable spatial density gradient. Feireisl et al.\ \cite{FNS11} 
proved the existence of so-called suitable weak solutions  
satisfying a general relative energy inequality with respect to any
sufficiently regular pair of functions and concluded the weak--strong uniqueness property.

Compressible Euler equations with chemotaxis force have been introduced in \cite{SAGGPB03}
to describe early stages of vascologenesis. As remarked in \cite[Section 3]{AGS04},
the fluid equations may also include viscous terms. This leads to chemotaxis compressible
Navier--Stokes equations, which have been analyzed in \cite{AiAl16} with the pressure
function $p(\rho)=\max\{0,\rho-\rho_c\}^\gamma$, where $\gamma>3$ and $\rho_c>0$ is the
so-called close-packing density. A viscoelastic mechanical interaction 
of the cells with the substratum was added to the compressible Euler equations in \cite{TAP06}.
Related models are the incompressible Navier--Stokes equations coupled to the
chemotaxis Keller--Segel system via the fluid velocity, proposed in \cite{TCDWK05}
and analyzed in, e.g., \cite{Win12}.

The paper is organized as follows. Section \ref{sec.ex} is devoted to the
proof of Theorem \ref{thm.ex}. The technical relative energy inequality \eqref{1.relEineq} is proved in Section \ref{sec.relE}. Based on this inequality, Theorem \ref{thm.wsu} is then shown in Section \ref{sec.wsu}. Finally, some auxiliary results are presented in Appendix \ref{sec.aux}.


\section{Global existence of solutions}\label{sec.ex}

In this section, we prove Theorem \ref{thm.ex}. For this, we proceed as in
\cite{FNP01} by constructing an approximate scheme based on a regularized system,
deriving uniform energy estimates, and assing to the de-regularization limit. For later use, we note the relations between the pressure $p(\rho)$ and the associated internal energy $\psi(\rho)=\rho\int_0^\rho s^{-2}p(s)ds$:
\begin{equation}\label{1.p}
  p(\rho) = \rho\psi'(\rho)-\psi(\rho), \quad \na p(\rho) = \rho\na\psi'(\rho)
	\quad\mbox{for smooth }\rho.
\end{equation} 

\subsection{Regularized system}\label{sec.galerkin}

We solve first the following regularized system for $\delta>0$, $\eps>0$, and $\beta>4$:
\begin{align}
  & \pa_t\rho + \diver(\rho v) = \eps\Delta\rho, \quad
	\pa_t c = \Delta c - c + \rho, \label{2.mass} \\
	& \pa_t(\rho v) + \diver(\rho v\otimes v) + \na p(\rho) + \eps\na\rho\cdot\na v
	+ \delta\na\rho^\beta \label{2.mom} \\
	&\phantom{xx}{}= \mu\Delta v + (\lambda+\mu)\na\diver v
	+ \rho\na c - \frac{\rho v}{\zeta}\quad\mbox{in }\Omega,\ t>0, \nonumber
\end{align}
subject to the initial and boundary conditions
\begin{align}
  & \rho(\cdot,0) = \rho^0_{\delta}, \quad v(\cdot,0) = v^0, \quad c(\cdot,0) = c^0
	\quad\mbox{in }\Omega, \label{2.ic} \\
	& \na\rho\cdot\nu = 0, \quad v = 0, \quad \na c\cdot\nu = 0 
	\quad\mbox{on }\pa\Omega,\ t>0, \label{2.bc}
\end{align}
where $\rho^0_{\delta}$ is a smooth strictly positive function such that
$\rho^0_{\delta}\to\rho^0$ strongly in $L^\gamma(\Omega)$. 
The artificial viscosity term $\eps\Delta\rho$
is balanced by the term $\eps\na\rho\cdot\na v$ in the momentum equation to control the
energy. The artificial pressure term $\delta\na\rho^\beta$ is needed
to derive an $L^{\gamma+\theta}(\Omega)$ estimate for the density with $\theta>0$.

The existence of strong solutions to
\eqref{2.mass}--\eqref{2.bc} was shown in \cite[Section 2]{FNP01} without the
chemotaxis term $\rho\na c$. Here, we sketch the proof for the problem including
the chemotaxis coupling. 
As in \cite{FNP01}, we use the Faedo--Galerkin method. Let $(\psi_n)$ be a sequence of
eigenfunctions of the Laplacian with homogeneous Dirichlet boundary conditions and let
$X_n=\operatorname{span}\{\psi_1,\ldots,\psi_n\}$. Then, following the proof of
\cite[Section~7.7]{NoSt04} or \cite[Chapter~7]{Fei04}, 
we obtain the existence of a unique local strong solution $(\rho_n,v_n,c_n)$ on $(0,T_n)$ such that $v_n\in C^1([0,T_n];X_n)$ and
\begin{align*}
  & \rho_n,\,\pa_t\rho_n,\,\na\rho_n,\,\na^2\rho_n,\,c_n,\,\pa_t c_n,\,\na c_n,\,\na^2 c_n
	\quad\mbox{are H\"older continuous on }\overline\Omega\times[0,T_n], \\
	& \rho_n(x,t)>0, \quad c_n(x,t)\ge 0\quad\mbox{for any }(x,t)\in\overline\Omega\times[0,T_n].
\end{align*}
To obtain global solutions, i.e.\ $T=T_n$, we derive an energy inequality for the
approximate system. 


\subsection{Energy inequality for the approximate system}

An energy-type inequality has been derived in \cite[Section 2.2]{AiAl16}. 
Here, we use a different energy functional by including the $H^1(\Omega)$ norm of $c$.
We show an inequality for the energies $E(\rho,v,c)$ and $H(\rho,v,c)$,
defined in \eqref{1.E} and \eqref{1.H}, respectively.

\begin{lemma}\label{lem.Hn}
Let $(\rho_n,v_n,c_n)$ be a strong solution to \eqref{2.mass}--\eqref{2.bc} constructed
in the previous subsection. Then there exists $C>0$ independent of $(n,\delta,\eps)$
such that for any $0<t<T_n$,
\begin{align*}
  E((&\rho_n,v_n,c_n)(t)) 
	+ \int_0^t\int_\Omega\big(\mu|\na v_n|^2 + (\lambda+\mu)|\diver v_n|^2\big)dxds
	+ \frac{4\delta\eps}{\beta}\int_0^t\int_\Omega|\na\rho_n^{\beta/2}|^2dxds \\
	&\phantom{xx}{}+ \frac{4\eps}{\gamma}\int_0^t\int_\Omega|\na\rho_n^{\gamma/2}|^2 dxds
	+ \bigg(1-\frac{\eps}{4}\bigg)\int_0^t\int_\Omega|\pa_s c_n|^2 dx ds
	+ \frac{1}{\zeta}\int_0^t\int_\Omega\rho_n|v_n|^2 dxds \\
	&\le E(\rho^0,v^0,v^0) + 2\eps\int_\Omega\rho^\gamma dx + C\eps, \\
  H((&\rho_n,v_n,c_n)(t)) + \int_0^t\int_\Omega\big(\mu|\na v_n|^2 
	+ (\lambda+\mu)|\diver v_n|^2\big)dxds \\
	&\phantom{xx}{}+ \frac{4\delta\eps}{\beta}\int_0^t\int_\Omega|\na\rho_n^{\beta/2}|^2 dxds
	+ \bigg(1-\frac{\eps}{4}\bigg)\int_\Omega\int_\Omega|\pa_s c_n|^2 dxds
	+ \frac{1}{\zeta}\int_0^t\int_\Omega \rho_n|v_n|^2 dxds \\
	&\le \big(E(\rho^0_{\delta},v^0,c^0) + C(\rho^0,v^0) + C\eps t\big)e^{C\eps t}.
\end{align*}
\end{lemma}

\begin{proof}
{\em Step 1: Energy inequality for $E$.}
We choose the test function $\psi'(\rho_n)-\frac12|v_n|^2+\delta\beta\rho_n^{\beta-1}/(\beta-1)$ 
in the weak formulation of the first equation in \eqref{2.mass} and the test function $v_n$ 
in the weak formulation of \eqref{2.mom}. Adding both equations and taking into account
\eqref{1.p}, some terms cancel, and we arrive after a standard computation at
\begin{align}\label{2.ener1}
  \frac{d}{dt}&\int_\Omega\bigg(\psi(\rho_n) + \frac12\rho_n|v_n|^2 
	+ \frac{\delta}{\beta-1}\rho_n^\beta\bigg)dx
	+ \int_\Omega\big(\mu|\na v_n|^2 + (\lambda+\mu)|\diver v_n|^2\big)dx \\
	&{}+ \frac{1}{\zeta}\int_\Omega\rho_n|v_n|^2 dx 
	+ \frac{4\delta\eps}{\beta}\int_\Omega|\na\rho_n^{\beta/2}|^2dx 
	+ \frac{4\eps}{\gamma}\int_\Omega|\na\rho_n^{\gamma/2}|^2 dx 
	= \int_\Omega\rho_n\na c_n\cdot v_n dx. \nonumber
\end{align}
We estimate the right-hand side by integrating by parts 
and using equation \eqref{2.mass} for $\rho_n$:
\begin{align}
   \int_\Omega\rho_n\na c_n\cdot v_n dx &= -\int_\Omega c_n\diver(\rho_n v_n)dx 
	= \int_\Omega c_n(\pa_t\rho_n-\eps\Delta\rho_n)dx \label{2.aux} \\
	&= \frac{d}{dt}\int_\Omega\rho_n c_n dx - \int_\Omega\rho_n\pa_t c_n dx
	- \eps\int_\Omega c_n\Delta\rho_n dx. \nonumber
\end{align}
Taking into account the second equation in \eqref{2.mass}, 
the second term on the right-hand side is written as
\begin{align*}
  -\int_\Omega\rho_n\pa_t c_n dx 
	&= -\int_\Omega(\pa_t c_n-\Delta c_n+c_n)\pa_t c_ndx \\
	&= -\int_\Omega|\pa_t c_n|^2 dx - \frac12\frac{d}{dt}\int_\Omega(|\na c_n|^2 + c_n^2)dx.
\end{align*}
Because of $\rho_n c_n\ge 0$, 
the last term on the right-hand side of \eqref{2.aux} becomes
\begin{align*}
  -\eps\int_\Omega c_n\Delta\rho_n dx
	&= -\eps\int_\Omega\rho_n\Delta c_n dx
	= -\eps\int_\Omega\rho_n(\pa_t c_n + c_n - \rho_n)dx \\
	&\le -\eps\int_\Omega\rho_n\pa_t c_n dx + \eps\int_\Omega\rho_n^2 dx 
	\le \frac{\eps}{4}\int_\Omega|\pa_t c_n|^2 dx + 2\eps\int_\Omega\rho_n^2 dx \\
	&\le \frac{\eps}{4}\int_\Omega|\pa_t c_n|^2 dx + 2\eps\int_\Omega\rho_n^\gamma dx
	+ C(\gamma,\Omega)\eps,
\end{align*}
where the last inequality follows from $\gamma\ge 2$. (We observe that at this point,
we can weaken the condition to $\gamma>8/5$ by using the Gagliardo--Nirenberg inequality 
and the estimate for $\|\na\rho_n^{\gamma/2}\|_{L^2(\Omega)}$ from \eqref{2.ener1}.)
We insert these estimates into \eqref{2.aux}:
\begin{align*}
  \int_\Omega\rho_n\na c_n\cdot v_n dx
	&\le	-\frac{d}{dt}\int_\Omega\bigg(\frac12(|\na c_n|^2
	+ c_n^2) - \rho_n c_n\bigg)dx - \int_\Omega|\pa_t c_n|^2 dx \\
	&\phantom{xx}{}+ \frac{\eps}{4}\int_\Omega|\pa_t c_n|^2 dx 
	+ 2\eps\int_\Omega\rho_n^\gamma dx + C(\gamma,\Omega)\eps.
\end{align*}
Therefore, \eqref{2.ener1} leads to
\begin{align}\label{2.eta}
  \frac{d}{dt}&\int_\Omega\bigg(\frac12\rho_n|v_n|^2 + \psi(\rho_n) 
	+ \frac12(|\na c_n|^2 + c_n^2) - \rho_n c_n 
	+ \frac{\delta}{\beta-1}\rho_n^\beta\bigg)dx \\
	&\phantom{xx}{}+ \int_\Omega\big(\mu|\na v_n|^2 + (\lambda+\mu)|\diver v_n|^2\big)dx
	+ \frac{4\delta\eps}{\beta}\int_\Omega|\na\rho_n^{\beta/2}|^2dx 
	+ \frac{4\eps}{\gamma}\int_\Omega|\na\rho_n^{\gamma/2}|^2 dx \nonumber \\
	&\phantom{xx}{}+ \bigg(1-\frac{\eps}{4}\bigg)|\pa_t c_n|^2 dx 
	+ \frac{1}{\zeta}\int_\Omega\rho_n|v_n|^2 dx
	\le 2\eps\int_\Omega\rho_n^\gamma dx + C\eps, \nonumber
\end{align}
where $C>0$ only depends on $\gamma$ and $\mbox{meas}(\Omega)$ but is independent of
$n$, $\delta$, and $\eps$. This proves the inequality for $E(\rho_n,v_n,c_n)$.

{\em Step 2: Energy inequality for $H$.}
We need to estimate $\int_\Omega\rho_n c_n dx$ in $E(\rho_n,v_n,c_n)$.
By Lemma \ref{lem.aux} in Appendix \ref{sec.aux}, applied to $m=\gamma$, 
$\kappa=1/(2(\gamma-1))$, and $\xi=1/4$,
\begin{equation}\label{2.aux2}
  \int_\Omega\rho_n c_n dx \le \frac{1}{2(\gamma-1)}\|\rho_n\|_{L^\gamma(\Omega)}^\gamma
	+ \frac14\|\na c_n\|_{L^2(\Omega)}^2 + C_1(\gamma)\|c_n\|_{L^1(\Omega)}^{C_2(\gamma)}.
\end{equation}
Equation \eqref{2.mass} implies that the mass is conserved, $\|\rho_n(t)\|_{L^1(\Omega)}
= \|\rho^0_{\delta}\|_{L^1(\Omega)}$ for $0<t<T_n$. Furthermore, by the second
equation in \eqref{2.mass},
$$
  \frac{d}{dt}\int_\Omega c_n dx = \int_\Omega\rho_n dx - \int_\Omega c_n dx.
$$
This is an ordinary differential equation for $t\mapsto\|c_n(t)\|_{L^1(\Omega)}$, and
a comparison principle as well as the nonnegativity of $c_n$ imply that
$$
  \|c_n(t)\|_{L^1(\Omega)} = \int_\Omega c_n dx \le\max\bigg\{\int_\Omega c^0dx,
	\int_\Omega\rho^0_{\delta}dx\bigg\} \le C,
$$
where $C>0$ is independent of $\delta$. Thus, we conclude from \eqref{2.aux2} 
and $\rho_n^\gamma/(2(\gamma-1)) = \frac12\psi(\rho_n)$ that
$$
  \int_\Omega\rho_n c_n dx \le \frac12\int_\Omega\psi(\rho_n) dx
	+ \frac14\|\na c_n\|_{L^2(\Omega)}^2 + C(\rho^0,c^0).
$$
It follows from the definitions of $E(\rho_n,v_n,c_n)$ and $H(\rho_n,v_n,c_n)$ that
\begin{align*}
  E(\rho_n,v_n,c_n) &\ge \int_\Omega\bigg(\frac12\psi(\rho_n) + \frac12\rho_n|v_n|^2 
	+ \frac14|\na c_n|^2 + \frac12 c_n^2\bigg)dx - C(\rho^0,c^0) \\
	&= H(\rho_n,v_n,c_n) - C(\rho^0,c^0).
\end{align*}
We insert these estimates in \eqref{2.eta} and integrate over $(0,t)$ for $0<t<T_n$:
\begin{align*}
  H((&\rho_n,v_n,c_n)(t)) + \int_0^t\int_\Omega\big(\mu|\na v_n|^2 
	+ (\lambda+\mu)|\diver v_n|^2\big)dxds 
	+ \frac{4\delta\eps}{\beta}\int_0^t\int_\Omega|\na\rho_n^{\beta/2}|^2 dxds \\
	&\phantom{xx}{}+ \frac{4\eps}{\gamma}\int_0^t\int_\Omega|\na\rho_n^{\gamma/2}|^2 dxds
	+ \bigg(1-\frac{\eps}{4}\bigg)\int_0^t\int_\Omega|\pa_s c_n|^2 dxds
	+ \frac{1}{\zeta}\int_0^t\int_\Omega \rho_n|v_n|^2 dxds \\
	&\le E(\rho^0_{\delta},v^0,c^0) + C\eps\int_0^t\int_\Omega H(\rho_n,v_n,c_n)dxds 
	+ C(\rho^0,c^0) + C\eps t,
\end{align*}
where we used $\int_\Omega\rho_n^\gamma dx\le CH(\rho_n,v_n,c_n)$. 
An application of Gronwall's lemma finishes the proof.
\end{proof}

Lemma \ref{lem.Hn} allows us to conclude as in \cite[Section 2.3]{FNP01} that
$T=T_n$. Moreover, it yields the following estimates uniform in $(n,\delta,\eps)$:
\begin{equation}\label{2.bounds}
\begin{aligned}
  (\rho_n) &\quad\mbox{is uniformly bounded in }L^\infty(0,T;L^\gamma(\Omega)), \\
  (\sqrt{\rho_n}v_n) &\quad\mbox{is uniformly bounded in }L^\infty(0,T;L^2(\Omega;\R^3)), \\
	(\na v_n) &\quad\mbox{is uniformly bounded in }L^2(0,T;L^2(\Omega;\R^{3\times 3})), \\
	(c_n) &\quad\mbox{is uniformly bounded in }L^\infty(0,T;H^1(\Omega))\cap
	H^1(0,T;L^2(\Omega)).
\end{aligned}
\end{equation}


\subsection{Limit $(n,\delta,\eps)\to(\infty,0,0)$}

The limit $n\to\infty$ can be performed as in \cite[Section 2.3]{AiAl16} via the
Aubin--Lions compactness lemma. This gives a solution $(\rho_\delta,v_\delta,c_\delta)$
to \eqref{2.mass}--\eqref{2.bc}. It satisfies the energy inequalities in Lemma \ref{lem.Hn}.
In particular, we conclude a uniform bound for $\rho_\delta$ 
in $L^\infty(0,T;L^\gamma(\Omega))$. By the existence theory for the compressible
Navier--Stokes equations, we can pass to the limit $(\delta,\eps)\to 0$;
see, e.g., \cite{Fei04,NoSt04}.  Indeed, to derive a uniform
estimate for the mass density in $L^{\gamma+\theta}(\Omega\times(0,T))$ 
for some $\theta>0$, we need to use the test function 
$$
  \phi_B = \mathcal{B}\bigg(\rho_\delta^\theta - \frac{1}{|\Omega|}
	\int_\Omega\rho_\delta^\theta dx\bigg),
$$
in the weak formulation of the approximate momentum equation,
where $\mathcal{B}$ is the Bogovskii operator \cite[Section 3.3.1.2]{NoSt04}.
Compared to the compressible Navier--Stokes equations, the momentum equation includes 
the chemotaxis term $\rho_\delta\na c_\delta$, which needs to be estimated. This means that
we need a bound for
\begin{equation}\label{2.I}
  I = \int_0^T\int_\Omega\rho_\delta\na c_\delta\cdot\phi_B dxdt.
\end{equation}
Using the second equation in \eqref{2.mass},
$$
  \rho_\delta\na c_\delta = (\pa_t c_\delta - \Delta c_\delta + c_\delta)\na c_\delta
	= (\pa_t c_\delta + c_\delta)\na c_\delta - \diver(\na c_\delta\otimes\na c_\delta)
	+ \frac12\na|\na c_\delta|^2,
$$
we can write $I=I_1+\cdots+I_4$, where
\begin{align*}
  I_1 &= \int_0^T\int_\Omega\pa_t c_\delta\na c_\delta\cdot\phi_B dxdt, &
	I_2 &= \int_0^T\int_\Omega c_\delta\na c_\delta\cdot\phi_B dxdt, \\
	I_3 &= \int_0^T\int_\Omega\na c_\delta\otimes\na c_\delta:\na\phi_B dxdt, &
	I_4 &= -\frac12\int_0^T\int_\Omega|\na c_\delta|^2 \diver\phi_B dxdt.
\end{align*}

We start with the term $I_1$. First, let $\gamma>2$. By parabolic regularity theory (see Theorem \ref{thm.maxreg} in the Appendix with $p=q=\gamma$), the continuous embedding $W^{2-2/\gamma,\gamma}(\Omega)\hookrightarrow W^{1,\gamma}(\Omega)$ and the second equation in \eqref{2.mass} yield
\begin{align}\label{2.estc}
  \|\na c_\delta\|_{L^2(0,T;L^\gamma(\Omega))} 
	&\le C\|c_\delta\|_{L^2(0,T;W^{2-2/\gamma,\gamma}(\Omega))} \\
	&\le C\big(\|\rho_\delta\|_{L^\infty(0,T;L^\gamma(\Omega))}
	+ \|c^0\|_{W^{2-2/\gamma,\gamma}(\Omega)}\big)
	\le C. \nonumber
\end{align}
Hence, using H\"older's inequality, the assumption $\gamma>2$, 
and the previous estimates as well as the uniform estimates from the energy inequality,
$$
  I_1 \le \|\pa_t c_\delta\|_{L^2(0,T;L^2(\Omega))}
	\|\na c_\delta\|_{L^2(0,T;L^\gamma(\Omega))}
	\|\phi_B\|_{L^\infty(0,T;L^{2\gamma/(\gamma-2)}(\Omega))}
	\le C\|\phi_B\|_{L^\infty(0,T;W^{1,r}(\Omega))},
$$
where $r=6\gamma/(5\gamma-6)$ is such that $W^{1,r}(\Omega)\hookrightarrow 
L^{2\gamma/(\gamma-2)}(\Omega)$. 
We deduce from the boundedness of 
$\mathcal{B}:L_0^r(\Omega)\to W_0^{1,r}(\Omega)$ for $1<r<\infty$,
where $L_0^r(\Omega)$ is the space of all $L^r(\Omega)$ functions $u$ satisfying $\int_\Omega udx=0$, that
\begin{align*}
	I_1 &\le C\bigg\|\rho_\delta^\theta - \frac{1}{|\Omega|}
	\int_\Omega\rho_\delta^\theta dx\bigg\|_{L^\infty(0,T;L^r(\Omega))}
	\le C\|\rho_\delta^\theta\|_{L^\infty(0,T;L^r(\Omega))} \\
	&\le C\|\rho_\delta\|_{L^\infty(0,T;L^{r\theta}(\Omega))}^\theta
	\le C\|\rho_\delta\|_{L^\infty(0,T;L^\gamma(\Omega))}^\theta\le C.
\end{align*}
The last but one step follows if $r\theta\le\gamma$, which requires the choice $0<\theta\le 5\gamma/6-1$, and the last step is a consequence of the energy inequality. 

Next, let $3/2<\gamma\le 2$. We apply Theorem \ref{thm.maxreg} with $p=2$, $q=\gamma$ to find that 
$$
  \|c_\delta\|_{L^2(0,T;W^{2,\gamma}(\Omega))}
  + \|\pa_t c_\delta\|_{L^2(0,T;L^\gamma(\Omega))}
  \le C\big(\|\rho_\delta\|_{L^2(0,T;L^\gamma(\Omega))}
  + \|c^0\|_{W^{1,\gamma}(\Omega)}\big) \le C.
$$
Hence, we deduce from the continuous embedding $W^{2,\gamma}(\Omega)\hookrightarrow W^{1,3\gamma/(3-\gamma)}(\Omega)$ that
\begin{align*}
  I_1 &\le \|\pa_t c_\delta\|_{L^2(0,T;L^\gamma(\Omega))}
  \|\na c_\delta\|_{L^2(0,T;L^{3\gamma/(3-\gamma)}(\Omega))}
  \|\phi_B\|_{L^\infty(0,T;L^{3\gamma/(4\gamma-6)}(\Omega))} \\
  &\le C\|\pa_t c_\delta\|_{L^2(0,T;L^\gamma(\Omega))}
    \|c_\delta\|_{L^2(0,T;W^{2,\gamma}(\Omega))}
    \|\phi_B\|_{L^\infty(0,T;W^{1,r}(\Omega))},
\end{align*}
where now $r=3\gamma/(5\gamma-6)$. We choose $\theta>0$ such that $r\theta\le\gamma$, which is equivalent to $\theta\le 5\gamma/3-2$, and we can choose $\theta>0$ satisfying this inequality. Then, arguing as in the case $\gamma\ge 2$,
$$
  I_1 \le C\|\rho_\delta\|_{L^\infty(0,T;L^{r\theta}(\Omega))}^\theta
  \le C\|\rho_\delta\|_{L^\infty(0,T;L^\gamma(\Omega))}^\theta
  \le C.
$$

For the term $I_3$, we consider again first the case $\gamma>2$:
\begin{align*}
  I_3 \le \|\na c_\delta\|_{L^2(0,T;L^\gamma(\Omega))}^2
  \|\na\phi_B\|_{L^\infty(0,T;L^r(\Omega))}
  \le C\|\rho_\delta\|_{L^\infty(0,T;L^{r\theta}(\Omega))}^\theta
  \le C,
\end{align*}
where $r=\gamma/(\gamma-2)$, and the last inequality follows if $r\theta\le\gamma$, which is equivalent to $\theta\le\gamma-2$.
If $3/2<\gamma\le 2$, we proceed similarly as for $I_1$:
\begin{align*}
  I_3 \le \|\na c_\delta\|_{L^2(0,T;L^{3\gamma/(3-\gamma)}(\Omega))}^2
  \|\na\phi_B\|_{L^\infty(0,T;L^r(\Omega))}
  \le C\|\rho_\delta\|_{L^\infty(0,T;L^{r\theta}(\Omega))}^\theta
    \le C,
\end{align*}
where $r=\gamma/(2\gamma-3)$ and we need $r\theta\le\gamma$ or, equivalently, $\theta\le 2\gamma-3$.

The term $I_2$ is estimated in a similar way as $I_1$, and $I_4$ can be bounded as $I_3$. This shows that $I$ is bounded and provides a uniform estimate for $\rho_\delta$ in $L^{\gamma+\theta}(\Omega\times(0,T))$.
Now we can proceed as in \cite[Section 7.3]{NoSt04} to prove the strong convergence of the pressure.


\begin{remark}[On the condition on $\gamma$ in \cite{AiAl16}]\label{rem.AiAl16}\rm
A\"{\i}ssa and Alexandre have estimated the term $I$, defined in \eqref{2.I}, in a different way.
They used the test function $\psi_B=\mathcal{B}(\rho_\delta - |\Omega|^{-1}
\int_\Omega\rho_\delta dx)$: 
\begin{align*}
  I &\le \|\rho_\delta\|_{L^2(\Omega\times(0,T))}\|\na c_\delta\|_{L^2(\Omega\times(0,T))}
	\|\psi_B\|_{L^\infty(0,T;L^\infty(\Omega))} \\
	&\le \|\rho_\delta\|_{L^2(\Omega\times(0,T))}\|\na c_\delta\|_{L^2(\Omega\times(0,T))}
	\|\rho_\delta\|_{L^\infty(0,T;L^r(\Omega))},
\end{align*}
which is a consequence of the estimate $\|\mathcal{B}(f)\|_{L^\infty(\Omega)}
\le C\|\mathcal{B}(f)\|_{W^{1,r}(\Omega)}\le C\|f\|_{L^r(\Omega)}$ choosing $r>d=3$.
Thus, the technique of \cite{AiAl16} only works for $\gamma>3$.
\qed
\end{remark}

\begin{remark}[On the condition $\gamma>8/5$]\label{rem.gamma}\rm
This restriction is needed to estimate the integral $\int_\Omega \rho_n c_n dx$ by means of Lemma \ref{lem.aux}. The idea is to obtain ``small'' terms that can be absorbed by the left-hand side of the energy inequality \eqref{2.eta} and terms that can be controlled (the $L^1(\Omega)$ norm of $c_n$). By the H\"older and Gagliardo--Nirenberg inequalities, we may estimate in a different way:
$$
  \int_\Omega \rho_n c_n dx \le\|\rho_n\|_{L^\gamma(\Omega)}
  \|c_n\|_{L^{\gamma/(\gamma-1)}(\Omega)}
  \le C\|\rho_n\|_{L^\gamma(\Omega)}\|c_n\|_{W^{2,\gamma}(\Omega)}^\theta
  \|c_n\|_{L^1(\Omega)}^{1-\theta},
$$
where $\theta=3/(5\gamma-3)\in(0,1)$ (which requires that $\gamma>6/5$).
It follows from the maximal regularity result of Theorem \ref{thm.maxreg} that 
$$
  \int_\Omega \rho_n c_n dx \le C\|\rho_n\|_{L^\gamma(\Omega)}
  (\|\rho_n\|_{L^\gamma(\Omega)}+1)^\theta,
$$
where $C>0$ depends on $\|c^0\|_{L^1(\Omega)}$.
We can conclude if $1+\theta<\gamma$, which is equivalent to $\gamma>8/5$. Thus, even taking into account maximal regularity does not improve the range for $\gamma$.
\end{remark}

\begin{remark}[Improving the condition on $\gamma$]\label{rem.improve}\rm
When the dynamics of the chemical concentration is much faster than that one of the
cell density, we can neglect the time derivative of the concentration in \eqref{1.chem},
and $c_\delta$ solves $0=\Delta c_\delta-c_\delta+\rho_\delta$ in $\Omega$. 
In this situation, we are able to weaken the condition on $\gamma$ to $\gamma>3/2$. 
Indeed, estimate \eqref{2.estc} still holds for the elliptic problem. 
The embedding $W^{1,\gamma}(\Omega)\hookrightarrow L^{3\gamma/(3-\gamma)}(\Omega)$ 
for $\gamma<3$ then shows that
$$
  \|\na c_\delta\|_{L^\infty(0,T;L^{3\gamma/(3-\gamma)}(\Omega))}
	\le C\|c_\delta\|_{L^\infty(0,T;W^{2,\gamma}(\Omega))}
	\le C\|\rho_\delta\|_{L^\infty(0,T;L^\gamma(\Omega))} \le C.
$$
Hence, using H\"older's inequality, we estimate
\begin{align*}
  I_3 &\le \|\na c_\delta\|^2_{L^\infty(0,T;L^{3\gamma/(3-\gamma)}(\Omega))}
	\|\na\phi_B\|_{L^\infty(0,T;L^{3\gamma/(5\gamma-6)}(\Omega))} \\
	&\le C\|\rho_\delta\|_{L^\infty(0,T;L^\gamma(\Omega))}^2 
	\|\rho_\delta^\theta\|_{L^\infty(0,T;L^{3\gamma/(5\gamma-6)}(\Omega))}
	\le C+C\|\rho_\delta\|_{L^\infty(0,T;L^\gamma(\Omega))}^{2+\theta} \le C,
\end{align*}
provided that $0<\theta< (5\gamma-6)/3$.
The terms $I_2$ and $I_4$ are estimated in a similar way and $I_1=0$, thus proving that $I$ is
bounded. This yields a uniform estimate for $\rho_\delta$ in $L^{\gamma+\theta}(\Omega
\times(0,T))$ for $\gamma>3/2$ according to the theory of the compressible Navier--Stokes equations.
\qed\end{remark}


\section{Relative energy inequality}\label{sec.relE}

We show a relative energy inequality for smooth functions.

\begin{lemma}[Relative energy inequality]\label{lem.relE}
Let $(\rho,v,c)$ be a smooth solution to \eqref{1.mass}--\eqref{1.bc} and let $(r,u,z)$
be smooth functions satisfying $r>0$ in $\overline\Omega\times[0,T]$ and $u=0$ on $\pa\Omega$. 
Then the relative energy inequality \eqref{1.relEineq} holds for $0<t<T$ with
\begin{align*}
  R(\rho,v,c|r,u,z) &= -\int_\Omega p(\rho|r)\diver u dx 
	- \int_\Omega\psi''(r)(\rho-r)g dx - \int_\Omega h\pa_t(c-z)dx \nonumber \\
	&\phantom{xx}{}- \int_\Omega\na(c-z)\cdot((\rho-r)u)dx
	+ \int_\Omega(c-z)g dx \nonumber \\
	&\phantom{xx}{}
	- \int_\Omega\rho(v-u)\otimes(v-u):\na u dx - \frac{1}{\zeta}\int_\Omega\rho|v-u|^2 dx 
  \nonumber \\
	&\phantom{xx}{}
	- \int_\Omega\bigg(\frac{\rho-r}{r}\big(\mu\Delta u	+ (\lambda+\mu)\na\diver u\big)
	+ \rho f\bigg)\cdot(v-u)dx, \nonumber
\end{align*}
where
\begin{align}
  & f = \pa_t u + u\cdot\na u + \frac{1}{r}\na p(r) - \frac{1}{r}\big(\mu\Delta u
	+ (\lambda+\mu)\na\diver u\big) - \na z + \frac{u}{\zeta}, \label{3.f} \\
	& g = \pa_t r + \diver(ru), \quad h = \pa_t z - \Delta z + z - r. \label{3.gh}
\end{align}
\end{lemma}

We prove in Section \ref{sec.wsu} that the relative energy inequality \eqref{1.relEineq} holds for finite energy weak solutions $(\rho,v,c)$ and $(\bar\rho,\bar{v},\bar{c})$, where $(\bar\rho,\bar{v})$ satisfies \eqref{1.regul}. The proof of \eqref{1.relEineq} follows the lines of \cite[Section 3.2]{FNS11}, but some steps are different due to the additional chemotaxis force. For this reason, and for the convenience of the reader, we present a full proof.

\begin{proof}
Let $(r_m,u_m,z_m)_{m\in\N}$ be smooth functions satisfying $r_m>0$ in 
$\overline\Omega\times[0,T]$, $v_m\in C^1([0,T];$ $X_m)$, and $v_m=0$ on $\pa\Omega$
such that $(r_m,u_m,z_m)\to (r,z,u)$ as $m\to\infty$ in a sense made precise in Step 3 below.
Here, $X_m$ is the Faedo--Galerkin space defined in Section \ref{sec.galerkin}.
We introduce
\begin{align}
  & f_m = \pa_t u_m + u_m\cdot\na u_m + \frac{1}{r_m}\na p(r_m)
	- \frac{1}{r_m}\big(\mu\Delta u_m + (\lambda+\mu)\na\diver u_m\big) \label{fm} \\
	&\phantom{xxxx}{}- \na z_m + \frac{u_m}{\zeta}, \nonumber \\
	& g_m = \pa_t r_m + \diver(r_mu_m), \quad 
	h_m = \pa_t z_m - \Delta z_m + z_m - r_m. \label{gmhm}
\end{align}
Then $(f_m,g_m,h_m)\to(f,g,h)$ as $m\to\infty$ in the sense of
distributions, where $(f,g,h)$ is defined in \eqref{3.f}--\eqref{3.gh}. Finally, 
let $(\rho_n,v_n,c_n)$ be a Galerkin solution to \eqref{2.mass}--\eqref{2.bc}.
We compute in the following the approximate relative energy inequality.

{\em Step 1: Time derivative of the relative kinetic energy.}
We derive an equation for the time evolution of the relative kinetic energy
$\frac12\int_\Omega\rho_n|v_n-u_m|^2 dx$. It follows from the approximative mass 
balance equation \eqref{2.mass} that
\begin{align}
  \frac12\frac{d}{dt}&\big(\rho_n|v_n-u_m|^2\big)
	= -\frac12\big(\diver(\rho_n v_n) - \eps\Delta\rho_n\big)|v_n-u_m|^2
	+ \rho_n\pa_t(v_n-u_m)\cdot(v_n-u_m) \nonumber \\
	&= -\frac12\diver\big(\rho_n v_n|v_n-u_m|^2\big)
	+ \rho_n v_n\cdot\na(v_n-u_m)\cdot(v_n-u_m) \label{aux1} \\
	&\phantom{xx}{}+ \rho_n\pa_t(v_n-u_m)\cdot(v_n-u_m) + \frac{\eps}{2}\Delta\rho_n|v_n-u_m|^2.
	\nonumber
\end{align}
Since $\rho_n(\pa_t v_n+v_n\cdot\na v_n)=\pa_t(\rho_n v_n)+\diver(\rho_n v_n\otimes v_n)
- \eps\Delta\rho_n v_n$, the second and third terms on the right-hand side are written as
\begin{align*}
  \rho_n & v_n\cdot\na(v_n-u_m)\cdot(v_n-u_m) + \rho_n\pa_t(v_n-u_m)\cdot(v_n-u_m) \\
	&= \rho_n(\pa_t v_n+v_n\cdot\na v_n)\cdot(v_n-u_m) 
	- \rho_n(\pa_t u_m + v_n\cdot\na u_m)\cdot(v_n-u_m) \\
	&= \big(\pa_t(\rho_n v_n) + \diver(\rho_n v_n\otimes v_n)\big)\cdot(v_n-u_m)
	- \eps\Delta\rho_n v_n\cdot(v_n-u_m) \\
	&\phantom{xx}{}- \rho_n(\pa_t u_m+u_m\cdot\na u_m)\cdot(v_n-u_m)
	- \rho_n(v_n-u_m)\cdot\na u_m\cdot(v_n-u_m).
\end{align*}
We insert this expression into \eqref{aux1}, integrate over $\Omega$, and replace
$\pa_t(\rho_n v_n)+\diver(\rho_n v_n\otimes v_n)$ by the momentum equation \eqref{2.mom}:
\begin{align}
  \frac12\frac{d}{dt}&\int_\Omega\rho_n|v_n-u_m|^2dx
	= \frac{\eps}{2}\int_\Omega\Delta\rho_n|v_n-u_m|^2 dx
	- \int_\Omega\na(p(\rho_n)+\delta\rho_n^\beta)\cdot(v_n-u_m)dx \nonumber \\
	&\phantom{xx}{}- \int_\Omega\big(\mu\na v_n\cdot\na(v_n-u_m) 
	+ (\lambda+\mu)\diver v_n\diver(v_n-u_m)\big)dx \label{kinetic} \\
	&\phantom{xx}{}- \eps\int_\Omega\na\rho_n\cdot\na v_n\cdot(v_n-u_m)dx
	+ \int_\Omega\rho_n\na c_n\cdot(v_n-u_m)dx \nonumber \\
	&\phantom{xx}{}- \frac{1}{\zeta}\int_\Omega\rho_n v_n\cdot(v_n-u_m)dx
	- \eps\int_\Omega\Delta\rho_n v_n\cdot(v_n-u_m)dx \nonumber \\
	&\phantom{xx}{}- \int_\Omega\rho_n(\pa_t u_m+u_m\cdot\na u_m)\cdot(v_n-u_m)dx
	- \int_\Omega\rho_n(v_n-u_m)\cdot\na u_m\cdot(v_n-u_m)dx. \nonumber
\end{align}
We wish to reformulate the last but one term in the previous equality. For this, 
we add and subtract $r_m$ and replace $r_m(\pa_t u_m+u_m\cdot\na u_m)$ by \eqref{fm}:
\begin{align*}
  -\int_\Omega&\rho_n(\pa_t u_m+u_m\cdot\na u_m)\cdot(v_n-u_m)dx \\
	&= -\int_\Omega\bigg(1 + \frac{\rho_n-r_m}{r_m}\bigg)\big(r_m(\pa_t u_m+ u_m\cdot\na u_m)
	\big)\cdot(v_n-u_m)dx \\
	&= \int_\Omega\bigg(1 + \frac{\rho_n-r_m}{r_m}\bigg)\bigg(\na p(r_m) - r_m\na z_m 
	+ \frac{r_mu_m}{\zeta} - r_mf_m\bigg)\cdot(v_n-u_m)dx \\
	&\phantom{xx}{}- \int_\Omega\bigg(1 + \frac{\rho_n-r_m}{r_m}\bigg)\big(\mu\Delta u_m
	+ (\lambda+\mu)\na\diver u_m\big)\cdot(v_n-u_m)dx.
\end{align*}
Then, after a computation, \eqref{kinetic} becomes
\begin{align}\label{kinetic2}
  \frac12\frac{d}{dt}&\int_\Omega\rho_n|v_n-u_m|^2dx 
	= -\int_\Omega\bigg(\na p(\rho_n) - \frac{\rho_n}{r_m}\na p(r_m)\bigg)\cdot(v_n-u_m)dx \\
	&\phantom{xx}{}- \delta\int_\Omega\na\rho_n^\beta\cdot(v_n-u_m)dx
	+ \eps\int_\Omega\na\rho_n\cdot\na u_m\cdot(v_n-u_m)dx \nonumber \\
	&\phantom{xx}{}- \int_\Omega\big(\mu|\na(v_n-u_m)|^2 + (\lambda+\mu)|\diver(u_n-u_m)|^2
	\big)dx \nonumber \\
	&\phantom{xx}{}- \int_\Omega\rho_n(v_n-u_m)\otimes(v_n-u_m):\na u_m dx
	+  \int_\Omega\rho_n\na(c_n-z_m)\cdot(v_n-u_m)dx \nonumber \\
	&\phantom{xx}{}- \frac{1}{\zeta}\int_\Omega\rho_n|v_n-u_m|^2 dx
	- \int_\Omega\rho_n f_m\cdot(v_n-u_m)dx \nonumber \\
	&\phantom{xx}{}- \int_\Omega\frac{\rho_n-r_m}{r_m}\big(\mu\Delta u_m + (\lambda+\mu)
	\na\diver u_m\big)\cdot(v_n-u_m)dx. \nonumber 
\end{align}
We rewrite the first, second, and sixth terms on the right-hand side of \eqref{kinetic}.

{\em Step 2a: Reformulation of the pressure term.}
Observing that $p'(z)=z\psi''(z)$ for $z\ge 0$ (see \eqref{1.p})
and that $\rho_m-r_m$ satisfies
$$
  \pa_t(\rho_n-r_m) + \diver\big((\rho_n-u_m)u_m + \rho_n(v_n-u_m)\big)
	= \eps\Delta\rho_n - g_m,
$$
we can write the first term on the right-hand side of \eqref{kinetic2} as
\begin{align}
  -\int_\Omega&\bigg(\na p(\rho_n) - \frac{\rho_n}{r_m}\na p(r_m)\bigg)\cdot(v_n-u_m)dx 
	= \int_\Omega\rho_n\na\big(\psi'(\rho_n)-\psi'(r_m)\big)\cdot(v_n-u_m)dx \nonumber \\
	&= -\int_\Omega\big(\psi'(\rho_n)-\psi'(r_m)\big)\diver\big(\rho_n(v_n-u_m)\big)dx 
	\label{potential} \\
	&= -\int_\Omega\big(\psi'(\rho_n)-\psi'(r_m)\big)\pa_t(\rho_n-r_m)dx \nonumber \\
	&\phantom{xx}{}
	- \int_\Omega\big(\psi'(\rho_n)-\psi'(r_m)\big)\diver((\rho_n-r_m)u_m)dx \nonumber \\
	&\phantom{xx}{}+ \int_\Omega\big(\psi'(\rho_n)-\psi'(r_m)\big)(\eps\Delta\rho_n-g_m)dx.
	\nonumber
\end{align}
Taking into account that the evolution of the relative internal energy is given by
\begin{align*}
  \pa_t\psi(\rho_n|r_m) &= \pa_t\big(\psi(\rho_n)-\psi(r_m)-\psi'(r_m)(\rho_n-r_m)\big) \\
	&= \big(\psi'(\rho_n)-\psi'(r_m)\big)\pa_t\rho_n - \psi''(r_m)\pa_t r_m(\rho_n-r_m),
\end{align*}
the first term on the right-hand side of \eqref{potential} is reformulated as
\begin{align*}
  -&\int_\Omega\big(\psi'(\rho_n)-\psi'(r_m)\big)\pa_t(\rho_n-r_m)dx \\
	&= -\int_\Omega\big(\psi'(\rho_n)-\psi'(r_m)\big)\pa_t\rho_n dx 
	+ \int_\Omega\big(\psi'(\rho_n)-\psi'(r_m)\big)\pa_t r_m dx \\
	&= -\int_\Omega\bigg(\frac{d}{dt}\psi(\rho_n|r_m) + \psi''(r_m)\pa_t r_m(\rho_n-r_m)\bigg)dx 
	+ \int_\Omega\big(\psi'(\rho_n)-\psi'(r_m)\big)\pa_t r_m dx \\
	&= -\frac{d}{dt}\int_\Omega\psi(\rho_n|r_m)dx
	+ \int_\Omega\big(\psi'(\rho_n)-\psi'(r_m)-\psi''(r_m)(\rho_n-r_m)\big)
	(g_m-\diver(r_mu_m))dx,
\end{align*}
where we used definition \eqref{gmhm} of $g_m$ in the last step. 
Integrating by parts to get rid of the
divergence, inserting the corresponding expression into \eqref{potential}, and observing
that the integral over $(\psi'(\rho_n)-\psi'(r_m))g_m$ cancels with the
corresponding expression in \eqref{potential}, we find that
\begin{align*}
  -\int_\Omega&\bigg(\na p(\rho_n) - \frac{\rho_n}{r_m}\na p(r_m)\bigg)\cdot(v_n-u_m)dx 
	= -\frac{d}{dt}\int_\Omega\psi(\rho_n|r_m)dx \\
	&{}+ \int_\Omega\big\{\na\big(\psi'(\rho_n)-\psi'(r_m)\big)\cdot(\rho_n u_m)
	- \na\big(\psi''(r_m)(\rho_n-r_m)\big)\cdot(r_mu_m)\big\}dx \\
	&{}+ \eps\int_\Omega\big(\psi'(\rho_n)-\psi'(r_m)\big)\Delta\rho_n dx
	- \int_\Omega\psi''(r_m)(\rho_n-r_m)g_m dx.
\end{align*}
We claim that the second term on the right-hand side can be formulated in terms of
the relative pressure $p(\rho_n|r_m) = p(\rho_n)-p(r_m)-p'(r_m)(\rho_n-r_m)$. It follows
from \eqref{1.p} that $\na p(\rho_n)=\rho_n\na\psi'(\rho_n)$, 
$\na p'(r_m) = \na\psi'(r_m)+r_m\na\psi''(r_m)$ and hence,
\begin{align*}
  \na p(\rho_n|r_m) &= \rho_n\na\psi'(\rho_n) - r_m\na\psi'(r_m)
	- \na\big(r_m\psi''(r_m)(\rho_n-r_m)\big) \\
	&= \rho_n\na\big(\psi'(\rho_n)-\psi'(r_m)\big) - r_m\na\big(\psi''(r_m)(\rho_n-r_m)\big)
\end{align*}
and consequently,
\begin{align*}
  \int_\Omega&\big\{\na\big(\psi'(\rho_n)-\psi'(r_m)\big)\cdot(\rho_n u_m)
	- \na\big(\psi''(r_m)(\rho_n-r_m)\big)\big\}\cdot(r_mu_m)dx \\
	&= \int_\Omega\na p(\rho_n|r_m)\cdot u_m dx = -\int_\Omega p(\rho_n|r_m)\diver u_m dx.
\end{align*}
Therefore, 
\begin{align}\label{2a}
  -\int_\Omega&\bigg(\na p(\rho_n) - \frac{\rho_n}{r_m}\na p(r_m)\bigg)\cdot(v_n-u_m)dx 
	= -\frac{d}{dt}\int_\Omega\psi(\rho_n|r_m)dx \\
	&{} - \int_\Omega p(\rho_n|r_m)\diver u_m dx 
	+ \eps\int_\Omega\big(\psi'(\rho_n)-\psi'(r_m)\big)\Delta\rho_n dx \nonumber \\
	&{}- \int_\Omega\psi''(r_m)(\rho_n-r_m)g_m dx. \nonumber
\end{align}

{\em Step 2b: Reformulation of the chemotaxis term.}
We reformulate the sixth term on the right-hand side of \eqref{kinetic2} by
integrating by parts and using the mass balances \eqref{2.mass} and \eqref{gmhm}:
\begin{align}\label{chemo}
  \int_\Omega&\rho_n\na(c_n-z_m)\cdot(v_n-u_m)dx 
	= -\int_\Omega(c_n-z_m)\diver\big(\rho_n(v_n-u_m)\big)dx \\
	&= \int_\Omega(c_n-z_m)\diver\big(-\rho_nv_n + (\rho_n-r_m)u_m + r_mu_m\big)dx \nonumber \\
	&= \int_\Omega(c_n-z_m)\big(\pa_t(\rho_n-r_m) + \diver((\rho_n-r_m)u_m) 
	- \eps\Delta\rho_n + g_m\big)dx \nonumber \\
	&= \frac{d}{dt}\int_\Omega(c_n-z_m)(\rho_n-r_m)dx - \int_\Omega(\rho_n-r_m)\pa_t(c_n-z_m)dx 
	\nonumber \\
	&\phantom{xx}{}- \int_\Omega\na(c_n-z_m)\cdot((\rho_n-r_m)u_m)dx
	- \int_\Omega(c_n-z_m)(\eps\Delta\rho_n-g_m) dx. \nonumber 
\end{align}
In view of the second equation in \eqref{gmhm}, we have
$$
  \rho_n-r_m = \pa_t(c_n-z_m) - \Delta(c_n-z_m) + (c_n-z_m) + h_m.
$$
We insert this expression into the second term on the right-hand side of \eqref{chemo}:
\begin{align}\label{2b}
  \int_\Omega&\rho_n\na(c_n-z_m)\cdot(v_n-u_m)dx 
	= \frac{d}{dt}\int_\Omega(c_n-z_m)(\rho_n-r_m)dx 
	- \int_\Omega|\pa_t(c_n-z_m)|^2 dx \\
	&{}- \frac12\frac{d}{dt}\int_\Omega\bigg(|\na(c_n-z_m)|^2 + (c_n-z_m)^2\bigg)dx
	- \int_\Omega h_m\pa_t(c_n-z_m)dx \nonumber \\
	&{}- \int_\Omega\na(c_n-z_m)\cdot((\rho_n-r_m)u_m)dx
	- \int_\Omega(c_n-z_m)(\eps\Delta\rho_n-g_m) dx. \nonumber
\end{align}

{\em Step 2c: Reformulation of the artificial pressure term.}
We rewrite the second term on the right-hand side of \eqref{kinetic2} by
integrating by parts and using the mass balance equation \eqref{2.mass}:
\begin{align}\label{2c}
  -\delta\int_\Omega&\na\rho_n^\beta\cdot(v_n-u_m)dx
	= -\beta\delta\int_\Omega\rho_n^{\beta-2}\na\rho_n\cdot(\rho_n v_n)dx
	- \delta\int_\Omega\rho_n^\beta\diver u_m dx \\
	&= \frac{\beta\delta}{\beta-1}\int_\Omega\rho_n^{\beta-1}\diver(\rho_n v_n)dx
	- \delta\int_\Omega\rho_n^\beta\diver u_m dx \nonumber \\
	&= \frac{\beta\delta}{\beta-1}\int_\Omega\rho_n^{\beta-1}(\eps\Delta\rho_n-\pa_t\rho_n)dx
	- \delta\int_\Omega\rho_n^\beta\diver u_m dx \nonumber \\
	&= -\frac{\delta}{\beta-1}\frac{d}{dt}\int_\Omega\rho_n^\beta dx
	- \beta\delta\int_\Omega\rho_n^{\beta-2}|\na\rho_n|^2 dx
	- \delta\int_\Omega\rho_n^\beta\diver u_m dx. \nonumber 
\end{align}

{\em Step 2d: Collecting the reformulations.}
We include the reformulations \eqref{2a}, \eqref{2b}, and \eqref{2c} into 
\eqref{kinetic2} to find that
\begin{align}\label{step2d}
  \frac{d}{dt}&\int_\Omega\bigg\{\frac12\rho_n|v_n-u_m|^2 + \psi(\rho_n|r_m)
	+ \frac12\bigg(|\na(c_n-z_m)|^2 + (c_n-z_m)^2\bigg) \\
	&\phantom{xx}{}- (c_n-z_m)(\rho_n-r_m)
	+ \frac{\delta}{\beta-1}\rho_n^\beta\bigg\}dx
	+ \int_\Omega|\pa_t(c_n-z_m)|^2dx \nonumber \\
	&\phantom{xx}{}+ \beta\delta\int_\Omega\rho_n^{\beta-2}|\na\rho_n|^2 dx 
	+ \int_\Omega\big(\mu|\na(v_n-u_m)|^2 
	+ (\lambda+\mu)|\diver(v_n-u_m)|^2\big)dx \nonumber \\
	&= -\int_\Omega p(\rho_n|r_m)\diver u_m dx 
	+ \eps\int_\Omega\big(\psi'(\rho_n)-\psi'(r_m)\big)\Delta\rho_n dx \nonumber \\
	&\phantom{xx}{}- \int_\Omega\psi''(r_m)(\rho_n-r_m)g_m dx - \int_\Omega h_m\pa_t(c_n-z_m)dx
	- \delta\int_\Omega\rho_n^\beta\diver u_m dx \nonumber \\
	&\phantom{xx}{}- \int_\Omega\na(c_n-z_m)\cdot((\rho_n-r_m)u_m)dx
	- \int_\Omega(c_n-z_m)(\eps\Delta\rho_n-g_m) dx \nonumber \\
	&\phantom{xx}{}+ \eps\int_\Omega\na\rho_n\cdot\na u_m\cdot(v_n-u_m)dx
	- \int_\Omega\rho_n(v_n-u_m)\otimes(v_n-u_m):\na u_m dx \nonumber \\
	&\phantom{xx}{}- \frac{1}{\zeta}\int_\Omega\rho_n|v_n-u_m|^2 dx 
	- \int_\Omega\rho_nf_m\cdot(v_n-u_m)dx \nonumber \\
	&\phantom{xx}{}- \int_\Omega\frac{\rho_n-r_m}{r_m}\big(\mu\Delta u_m 
	+ (\lambda+\mu)\na\diver u_m\big)\cdot(v_n-u_m)dx. \nonumber 
\end{align}

{\em Step 3: Limit $(n,m)\to\infty$ and $(\delta,\eps)\to 0$.} 
As mentioned in \cite[Section 3.3]{FNS11}, the limit in the approximate relative energy 
inequality \eqref{step2d} follows step by step the existence proof in \cite[Chapter 7]{Fei04} 
or \cite[Chapter 7]{NoSt04}. In particular, we perform first the limit $n\to\infty$ in the
Faedo--Galerkin approximation $(\rho_n,v_n,c_n)\to (\rho_{\eps,\delta},v_{\eps,\delta},
c_{\eps,\delta})$. Then the functions $(r_m,u_m,z_m)$ are replaced by smooth functions
$(r,u,z)$ using a density argument. Third, we pass to the limit 
$(\rho_{\eps,\delta},v_{\eps,\delta},c_{\eps,\delta})\to(\rho_\delta,v_\delta,c_\delta)$
as $\eps\to 0$ and $(\rho_\delta,v_\delta,c_\delta)\to(\rho,v,c)$ as $\delta\to 0$.

In view of the bounds
\eqref{2.bounds}, we can pass to the limit $n\to\infty$ and $(\delta,\eps)\to 0$
in \eqref{step2d}. We assume that $(r_m,u_m,z_m)$ converges to $(r,u,z)$ 
as $m\to\infty$ in such a way that the limit $m\to\infty$ in \eqref{step2d} is possible.
Then some integrals in \eqref{step2d} disappear and we end up with
\begin{align*}
  \frac{d}{dt}&\int_\Omega\bigg(\psi(\rho|r) + \frac12\rho|v-u|^2
	+ \frac12\big(|\na(c-z)|^2 + (c-z)^2\big) - (\rho-r)(c-z)\bigg)dx \\
	&\phantom{xx}{} + \int_\Omega\big(\mu|\na(v-u)|^2 
	+ (\lambda+\mu)|\diver(v-u)|^2\big)dx + \int_\Omega|\pa_t(c-z)|^2dx \nonumber \\
	&= -\int_\Omega p(\rho|r)\diver u dx 
	- \int_\Omega\psi''(r)(\rho-r)g dx - \int_\Omega h\pa_t(c-z)dx \nonumber \\
	&\phantom{xx}{}- \int_\Omega\na(c-z)\cdot((\rho-r)u)dx
	+ \int_\Omega(c-z)g dx \nonumber \\
	&\phantom{xx}{}
	- \int_\Omega\rho(v-u)\otimes(v-u):\na u dx - \frac{1}{\zeta}\int_\Omega\rho|v-u|^2 dx 
  \nonumber \\
	&\phantom{xx}{}
	- \int_\Omega\bigg(\frac{\rho-r}{r}\big(\mu\Delta u	+ (\lambda+\mu)\na\diver u\big)
	+ \rho f\bigg)\cdot(v-u)dx. \nonumber
\end{align*}
This shows \eqref{1.relEineq} and finishes the proof.
\end{proof}


\section{Weak--strong uniqueness}\label{sec.wsu}

We split the proof in several steps.

{\em Step 1: Relative energy inequality.}
We claim that \eqref{1.relEineq} holds for finite energy weak solutions 
$(\rho,v,c)$ and $(\bar\rho,\bar v,\bar c)$, where $(\bar\rho,\bar v)$ satisfies 
the regularity \eqref{1.regul}. 
According to \cite[Section 4]{FNS11}, using a density argument,
the relative energy inequality \eqref{1.relEineq} still holds for
functions $(r,u)$ satisfying the following regularity conditions:
\begin{equation}\label{3.regul1}
\begin{aligned}
  & r\in C_{\rm weak}^0([0,T];L^\gamma(\Omega)), \quad
	u\in C^0_{\rm weak}([0,T];L^{2\gamma/(\gamma-1)}(\Omega;\R^3)), \\
	& |\na u|\in L^1(0,T;L^\infty(\Omega))\cap L^2(\Omega\times(0,T)), \quad
	u=0\mbox{ on }\pa\Omega, \\
	& \pa_t u\in L^1(0,T;L^{2\gamma/(\gamma-1)}(\Omega;\R^3))\cap
	L^2(0,T;L^{6\gamma/(5\gamma-6)}(\Omega;\R^3)), \\
	& |\na^2 u|\in L^1(0,T;L^{2\gamma/(2\gamma+1)}(\Omega))\cap
	L^2(0,T;L^{6/5}(\Omega)).
\end{aligned}
\end{equation}
Moreover, $r$ needs to be bounded away from zero and we require
$\na\psi'(r),\pa_t\psi'(r)\in L^1(0,T;$ $L^{2\gamma/(\gamma-1)}(\Omega))$.
An inspection of \eqref{1.relEineq} reveals that $z$ should satisfy
\begin{equation}\label{3.regul2}
\begin{aligned}
  & z\in C_{\rm weak}^0([0,T];H^1(\Omega))\cap H^1(0,T;L^2(\Omega)),\quad
	\Delta z\in L^2(0,T;L^2(\Omega)), \\
	& |\na z|\in L^1(0,T;L^{2\gamma/(2\gamma-1)}(\Omega))\cap
	L^2(0,T;L^{6\gamma/(5\gamma-6)}(\Omega)).
\end{aligned}
\end{equation}
It follows from \cite[Theorem 2.4]{FJN12} that \eqref{1.relEineq} still holds if
$(\rho,v,c)$ is a finite energy weak solution.

\begin{lemma}
Let $(\bar\rho,\bar v,\bar c)$ be a finite energy weak solution in the sense of
Definition \ref{def.weak} satisfying the additional regularity \eqref{1.regul}.
Furthermore, let $\bar{c}^0\in W^{2-2/\gamma,\gamma}(\Omega)$
and $\bar{c}^0\ge 0$ in $\Omega$.
Then $(\bar\rho,\bar v,\bar c)$ fulfills the regularity conditions
\eqref{3.regul1}--\eqref{3.regul2}.
\end{lemma}

\begin{proof}
Regularity \eqref{3.regul1} follows as in \cite[Section 4]{FNS11} from Sobolev embeddings.
Theorem \ref{thm.maxreg} in the Appendix shows that \eqref{3.regul2} is satisfied.
\end{proof}

The previous lemma shows that we can take $(r,u,z)=(\bar\rho,\bar{v},\bar{c})$
in \eqref{1.relEineq}. Then the remainder $R(\rho,v,c|\bar\rho,\bar{v},\bar{c})$ in Lemma \ref{lem.relE} simplies, since $f=0$ and $g=h=0$, and we find that
\begin{align}\label{3.R}
  & \int_0^t R(\rho,v,c|\bar\rho,\bar{v},\bar{c})ds = J_1+\cdots+ J_5, \quad\mbox{where} \\
	& J_1 = -\int_0^t\int_\Omega p(\rho|\bar\rho)\diver \bar{v}dxds, \nonumber \\
	& J_2 = -\int_0^t\int_\Omega\na(c-\bar{c})\cdot((\rho-\bar\rho)\bar{v})dxds, \nonumber \\
	& J_3 = -\int_0^t\int_\Omega\rho(v-\bar v)\otimes(v-\bar v):\na\bar v dxds, \nonumber \\
	& J_4 = -\frac{1}{\zeta}\int_0^t\int_\Omega\bar\rho|v-\bar v|^2 dxds, \nonumber \\
	& J_5 = -\int_0^t\int_\Omega\frac{\rho-\bar\rho}{\bar\rho}\big(\mu\Delta\bar v
	+ (\lambda+\mu)\na\diver\bar v\big)\cdot(v-\bar v)dxds. \nonumber 
\end{align} 

{\em Step 2: Estimation of $J_i$.}
The terms $J_i$ can be estimated as in \cite[Section 4.1]{FNS11}
except the new term $J_2$. Indeed, since $p(\rho)=(\gamma-1)\psi(\rho)$, 
we have $p(\rho|\bar\rho)=(\gamma-1)\psi(\rho|\bar\rho)$, showing that
$$
  J_1 \le C\int_0^t\int_\Omega\psi(\rho|\bar\rho)dxds,
$$
and H\"older's inequality gives
$$
  J_3 \le C\int_0^t\int_\Omega\rho|v-\bar v|^2 dxds,
$$
where $C>0$ depends on the $L^\infty(\Omega\times(0,T))$ norm of $\na\bar v$.
The term $J_4$ is nonpositive and can be neglected.
Formulas (4.13)--(4.14) in \cite{FNS11} lead to
$$
  J_5 \le \xi\int_0^t\|v-\bar v\|_{H^1(\Omega)}^2 ds	
	+ C(\xi)\int_0^t\int_\Omega\psi(\rho|\bar\rho)dxds,
$$
where $\xi>0$ is arbitrary and $C(\xi)>0$ depends on $\xi$ as well as
$\|\bar v\|_{L^\infty(0,t;W^{2,3}(\Omega))}$ 
and $\|\na^2\bar v\|_{L^\infty(0,t;L^q(\Omega))}$. At this point, we need the condition $q>3$.

To estimate the term $J_2$, which is not contained in \cite{FNS11}, we use 
equation \eqref{1.chem} for $c$ and integrate by parts:
\begin{align*}
  J_2 &= -\int_0^t\int_\Omega\na(c-\bar c)\cdot\bar v\big(\pa_t(c-\bar c) - \Delta(c-\bar c)
	+ (c-\bar c)\big)dx \\
	&= -\int_0^t\int_\Omega\pa_t(c-\bar c)\na(c-\bar c)\cdot\bar v dxds
	- \frac12\int_0^t\int_\Omega\na[(c-\bar c)^2]\cdot\bar v dxds \\
	&\phantom{xx}{}+ \int_0^t\int_\Omega\bigg(\diver\big(\na(c-\bar c)\otimes\na(c-\bar c)\big)
	- \frac12\na|\na(c-\bar c)|^2\bigg)\cdot\bar vdxds \\
	&= -\int_0^t\int_\Omega\pa_t(c-\bar c)\na(c-\bar c)\cdot\bar v dxds
	+ \frac12\int_0^t\int_\Omega(c-\bar c)^2\diver\bar v dxds \\
	&\phantom{xx}{}- \int_0^t\int_\Omega\bigg(\na(c-\bar c)\otimes(c-\bar c):\na\bar v
	- \frac12|\na(c-\bar c)|^2\diver\bar v\bigg)dxds.
\end{align*}
Then, by Young's inequality,
$$
  J_2 \le \frac12\int_0^t\int_\Omega|\pa_s(c-\bar c)|^2 dxds
	+ C\int_0^t\int_\Omega\big(|\na(c-\bar c)|^2 + (c-\bar c)^2\big)dxds,
$$
where $C>0$ depends on $\|\bar v\|_{L^\infty(0,T;W^{1,\infty}(\Omega))}$. 
Summarizing, it follows from \eqref{3.R} that
\begin{align*}
  \int_0^t & R(\rho,v,c|\bar\rho,\bar v,\bar c)ds
	\le \frac12\int_0^t\int_\Omega|\pa_s(c-\bar c)|^2 dxds 
	+ \xi\int_0^t\|v-\bar v\|_{H^1(\Omega)}^2 ds \\
	&{}+ C\int_0^t\int_\Omega\big(\psi(\rho|\bar\rho) + \rho|v-\bar v|^2 
	+ |\na(c-\bar c)|^2 + (c-\bar c)^2\big)dxds.
\end{align*}
The first term on the right-hand side can be absorbed by the last term on the
left-hand side of \eqref{1.relEineq}. The second term on the left-hand side of \eqref{1.relEineq}
can be bounded from below by Korn's inequality \cite[Lemma 2]{NoPo11} according to 
$$
  \int_\Omega\big(\mu|\na(v-\bar v)|^2 + (\lambda+\mu)|\diver(v-\bar v)|^2\big)dx
	\ge C_K\|v-\bar v\|_{H^1(\Omega)}^2,
$$
since $v=\bar v=0$ on $\pa\Omega$. Therefore, choosing $0<\xi<C_K$, \eqref{1.relEineq} yields
\begin{align}
  E((&\rho,v,c)(t)|(\bar\rho,\bar v,\bar c)(t))
	+ \frac12\int_0^t\int_\Omega|\pa_s(c-\bar c)|^2 dxds
	+ (C_K-\xi)\int_0^t\|v-\bar v\|_{H^1(\Omega)}^2 ds \nonumber \\
	&\le E(\rho_0,v_0,c_0|\bar\rho_0,\bar v_0,\bar c_0) \label{3.EH} \\
	&\phantom{xx}{}+ C\int_0^t\int_\Omega\big(\psi(\rho|\bar\rho) + \rho|v-\bar v|^2 
	+ |\na(c-\bar c)|^2 + (c-\bar c)^2\big)dxds \nonumber \\
	&\le E(\rho_0,v_0,c_0|\bar\rho_0,\bar v_0,\bar c_0)
	+ C\int_0^t H(\rho,v,c|\bar\rho,\bar{v},\bar{c})ds. \nonumber
\end{align}

{\em Step 3: Estimation of $\int_\Omega(\rho-\bar\rho)(c-\bar c)dx$.}
We use Lemma \ref{lem.aux} in Appendix \ref{sec.aux}
with $m=2$ and arbitrary $\kappa_1,\xi>0$ on the set $\{\rho\le R\}$ for some $R>0$:
\begin{align}\label{3.R1}
  \int_{\{\rho\le R\}}(\rho-\bar\rho)(c-\bar c)dx 
	&\le \kappa_1\|\rho-\bar\rho\|_{L^2(\Omega\cap\{\rho\le R\})}^2
	+ \xi\|\na(c-\bar c)\|_{L^2(\Omega)}^2 \\
	&\phantom{xx}{}+ C_1(\kappa_1,\xi)\|c-\bar c\|_{L^1(\Omega)}^{C_2(2)}, \nonumber
\end{align}
as well as with $m=\gamma$ (which requires $\gamma>8/5$)
and arbitrary $\kappa_2>0$ on the set $\{\rho>R\}$:
\begin{align}\label{3.R2}
  \int_{\{\rho>R\}}(\rho-\bar\rho)(c-\bar c)dx 
	&\le \kappa_2\|\rho-\bar\rho\|_{L^\gamma(\Omega\cap\{\rho>R\})}^\gamma
	+ \xi\|\na(c-\bar c)\|_{L^2(\Omega)}^2 \\
	&\phantom{xx}{}+ C_1(\kappa_2,\xi)\|c-\bar c\|_{L^1(\Omega)}^{C_2(\gamma)}. \nonumber
\end{align}
According to \cite[Lemma 2.4]{LaTz13}, there exist constants $C_3,C_4,c_p,C_p>0$ such that
\begin{equation*}
  \psi(\rho|\bar\rho) \ge \begin{cases}
	C_3|\rho-\bar\rho|^2 &\quad\mbox{if }0\le\rho\le R, \\
	C_4|\rho-\bar\rho|^\gamma &\quad\mbox{if }\rho>R,
	\end{cases}
\end{equation*}
as long as $c_p\le\bar\rho\le C_p$. Thus, we can replace the first term on the right-hand
sides of \eqref{3.R1} and \eqref{3.R2}, respectively, by 
$\kappa_1 C_3^{-1}\int_\Omega\psi(\rho|\bar\rho)dx$ and 
$\kappa_2 C_4^{-1} \int_\Omega\psi(\rho|\bar\rho)dx$,
and summing these inequalities, we obtain
\begin{align}\label{3.aux}
  \int_\Omega(\rho-\bar\rho)(c-\bar c)dx 
	&\le \bigg(\frac{\kappa_1}{C_3}+\frac{\kappa_2}{C_4}\bigg)\int_\Omega\psi(\rho|\bar\rho)dx
	+ 2\xi\|\na(c-\bar c)\|_{L^2(\Omega)}^2 \\
	&\phantom{xx}{}+ C_1(\kappa_1,\xi)\|c-\bar c\|_{L^1(\Omega)}^{C_2(2)}
	+ C_1(\kappa_2,\xi)\|c-\bar c\|_{L^1(\Omega)}^{C_2(\gamma)}. \nonumber
\end{align}

We wish to estimate the last two norms in terms of the initial data. 
To this end, we integrate \eqref{1.chem} and use the mass conservation 
$\|\rho(t)\|_{L^1(\Omega)}=\|\rho^0\|_{L^1(\Omega)}$ and
$\|\bar\rho(t)\|_{L^1(\Omega)}=\|\bar\rho^0\|_{L^1(\Omega)}$:
\begin{align*}
  \frac{d}{dt}\int_\Omega(c-\bar c)(t)dx 
	&= -\int_\Omega(c-\bar c)dx + \int_\Omega(\rho-\bar\rho)dx
	= -\int_\Omega(c-\bar c)dx + \int_\Omega(\rho^0-\bar\rho^0)dx.
\end{align*}
Gronwall's lemma yields
$$
  \int_\Omega(c-\bar c)(t)dx \le C\int_\Omega(c^0-\bar c^0)dx
	+ C\int_\Omega(\rho^0-\bar\rho^0)dx.
$$
The same argument with $\bar c-c$ then shows that
$$
  \|(c-\bar c)(t)\|_{L^1(\Omega)} \le C\big(\|c^0-\bar c^0\|_{L^1(\Omega)}
	+ \|\rho^0-\bar\rho^0\|_{L^1(\Omega)}\big).
$$
Hence, choosing $\kappa_1=C_3/4$, $\kappa_2=C_4/4$, and $\xi=1/8$, we deduce from
\eqref{3.aux} that
\begin{align*}
  \int_\Omega(\rho-\bar\rho)(c-\bar c)dx 
	&\le \frac12\int_\Omega\bigg(\psi(\rho|\bar\rho) + \frac12|\na(c-\bar c)|^2\bigg)dx \\
	&\phantom{xx}{}+ C_1(\kappa_1,\xi)\big(\|c^0-\bar c^0\|_{L^1(\Omega)}
	+ \|\rho^0-\bar\rho^0\|_{L^1(\Omega)}\big)^{C_2(2)} \\
	&\phantom{xx}{}+ C_1(\kappa_2,\xi)\big(\|c^0-\bar c^0\|_{L^1(\Omega)}
	+ \|\rho^0-\bar\rho^0\|_{L^1(\Omega)}\big)^{C_2(\gamma)}.
\end{align*}
The last two terms are bounded from above by 
$$
  C\big(\|c^0-\bar c^0\|_{L^1(\Omega)} + \|\rho^0-\bar\rho^0\|_{L^1(\Omega)}\big)^{C_5},
$$
where $C_5$ equals $C_2(2)$ or $C_2(\gamma)$ depending on whether
$\|c^0-\bar c^0\|_{L^1(\Omega)}+\|\rho^0-\bar\rho^0\|_{L^1(\Omega)}$
is smaller or larger than one. We conclude that
\begin{align}\label{3.aux2}
  \int_\Omega(\rho-\bar\rho)(c-\bar c)dx 
	&\le \frac12\int_\Omega\bigg(\psi(\rho|\bar\rho) + \frac12|\na(c-\bar c)|^2\bigg)dx \\
	&\phantom{xx}{}+ C\big(\|c^0-\bar c^0\|_{L^1(\Omega)}
	+ \|\rho^0-\bar\rho^0\|_{L^1(\Omega)}\big)^{C_5} \nonumber \\
	&\le \frac12 H(\rho,v,c|\bar\rho,\bar v,\bar c)
	+ C\big(\|c^0-\bar c^0\|_{L^1(\Omega)}
	+ \|\rho^0-\bar\rho^0\|_{L^1(\Omega)}\big)^{C_5}. \nonumber
\end{align}

{\em Step 4: End of the proof.} 
By \eqref{3.aux2}, the relative energy is bounded from below by
\begin{align*}
  E(\rho,v,c|\bar\rho,\bar v,\bar c)
	&\ge H(\rho,v,c|\bar\rho,\bar v,\bar c) - \int_\Omega(\rho-\bar\rho)(c-\bar c)dx \\
	&\ge \frac12H(\rho,v,c|\bar\rho,\bar v,\bar c)
	- C\big(\|c^0-\bar c^0\|_{L^1(\Omega)}
	+ \|\rho^0-\bar\rho^0\|_{L^1(\Omega)}\big)^{C_5}.
\end{align*}
We insert this estimate into \eqref{3.EH}:
\begin{align*}
  \frac12 H((&\rho,v,c)(t)|(\bar\rho,\bar v,\bar c)(t))
	+ \frac12\int_0^t\int_\Omega|\pa_s(c-\bar c)|^2 dxds 
	+ C\int_0^t\|v-\bar v\|_{H^1(\Omega)}^2 ds \\
	&\le E(\rho^0,v^0,c^0|\bar\rho^0,\bar v^0,\bar c^0)
	+ C\int_0^t H(\rho,v,c|\bar\rho,\bar v,\bar c)ds \nonumber \\
	&\phantom{xx}{}+ C\big(\|c^0-\bar c^0\|_{L^1(\Omega)}
	+ \|\rho^0-\bar\rho^0\|_{L^1(\Omega)}\big)^{C_5}. \nonumber
\end{align*}
An application of Gronwall's lemma gives
\begin{align*}
  H((\rho,v,c)(t)|(\bar\rho,\bar v,\bar c)(t))
	&\le C e^{Ct}\big\{E(\rho^0,v^0,c^0|\bar\rho^0,\bar v^0,\bar c^0) \\
	&\phantom{xx}{}+ \big(\|c^0-\bar c^0\|_{L^1(\Omega)}
	+ \|\rho^0-\bar\rho^0\|_{L^1(\Omega)}\big)^{C_5}\big\},
\end{align*}
and the choice $\rho^0=\bar\rho^0$, $v^0=\bar v^0$, $c^0=\bar c^0$ ends the proof.


\begin{appendix}
\section{Auxiliary results}\label{sec.aux}

\begin{lemma}\label{lem.aux}
Let $\Omega\subset\R^d$ be a bounded domain with $d\in\{2,3\}$ and let $m>2(d+1)/(d+2)$.
Furthermore, let $\kappa,\xi>0$. Then there exist constants $C_1(\kappa,\xi)>0$ and
$C_2(m)>0$ such that for all $\rho\in L^m(\Omega)$, $c\in H^1(\Omega)$,
$$
  \int_\Omega \rho cdx \le \kappa\|\rho\|_{L^m(\Omega)}^m + \xi\|\na c\|_{L^2(\Omega)}^2
	+ C_1(\kappa,\xi)\|c\|_{L^1(\Omega)}^{C_2(m)}.
$$ 
\end{lemma}

\begin{proof}
The proof of the lemma is contained in \cite[Appendix B]{Sug07} for solutions to the 
degenerate Keller--Segel equations. For clarity, we present the proof for general functions
$\rho$ and $c$.
We conclude from the interpolation inequality for Lebesgue spaces and Young's inequality that
for any $\kappa>0$,
$$
  \|\rho c\|_{L^1(\Omega)} \le \|\rho\|_{L^m(\Omega)}\|c\|_{L^{m/(m-1)}(\Omega)}
	\le \kappa\|\rho\|_{L^m(\Omega)}^m + C(\kappa)\|c\|_{L^{m/(m-1)}(\Omega)}^{m/(m-1)}.
$$
We estimate the second term on the right-hand side by applying the Gagliardo--Nirenberg
inequality with $\theta=2d/(m(d+2))$:
$$
  \|c\|_{L^{m/(m-1)}(\Omega)} \le C\|\na c\|_{L^2(\Omega)}^\theta\|c\|_{L^1(\Omega)}^{1-\theta}
	+ C\|c\|_{L^1(\Omega)}.
$$
Then, by Minkowski's and Young's inequality, for any $\eps>0$,
\begin{align*}
   \|\rho c\|_{L^1(\Omega)} &\le \kappa\|\rho\|_{L^m(\Omega)}^m 
	+ C(\kappa,m)\big(\|\na c\|_{L^2(\Omega)}^{m\theta/(m-1)}
	\|c\|_{L^1(\Omega)}^{m(1-\theta)/(m-1)}	+ \|c\|_{L^1(\Omega)}^{m/(m-1)}\big) \\
	&\le \kappa\|\rho\|_{L^m(\Omega)}^m + C(\kappa,m)\eps\|\na c\|_{L^2(\Omega)}^2 \\
	&\phantom{xx}{}+ C(\kappa,m,\eps)\big(\|c\|_{L^1(\Omega)}^{2m(1-\theta)/(2(m-1)-m\theta)}	
	+ \|c\|_{L^1(\Omega)}^{m/(m-1)}\big),
\end{align*}
which is possible since $m\theta/(m-1)<2$ is equivalent to $m>2(d+1)/(d+2)$. 
The lemma follows after choosing $\eps=\xi/C(\kappa,m)$, $C_1(\kappa,\xi)=C(\kappa,m,\eps)$,
and $C_2(m)=\max\{m/(m-1),2m(1-\theta)/(2(m-1)-m\theta)\}$.
\end{proof}

The following result concerns the maximal regularity of the solution to
\begin{align}
  & \pa_t u - \Delta u + u = f\quad\mbox{in }\Omega,\ t>0, \label{a.eq1} \\
	& \na u\cdot\nu=0\mbox{ on }\pa\Omega,\ t>0, \quad u(\cdot,0)=u^0\mbox{ in }\Omega,
	\label{a.eq2}
\end{align}
where $\Omega\subset\R^d$ ($d\ge 1$) is a bounded domain with $C^3$ boundary. We recall that $W^{2,-2/p,q}_\nu(\Omega)$ is the completion of the space of functions $w\in C^\infty(\overline\Omega)$ satisfying $\na w\cdot\nu=0$ on $\pa\Omega$ in the norm of $W^{2-2/p,q}(\Omega)$.
The theorem is a special case of \cite[Theorem 10.22]{FeNo09} or \cite[Lemma 7.37]{NoSt04}.

\begin{theorem}[Maximal regularity]\label{thm.maxreg}
Let $1<p,q<\infty$, $f\in L^p(0,T;L^q(\Omega))$, and let $u^0\in W^{2-2/p,q}_\nu(\Omega)$. 
Then there exists a unique solution $u$ to \eqref{a.eq1}--\eqref{a.eq2} satisfying
$$
  u\in L^p(0,T;W^{2,q}(\Omega))\cap W^{1,p}(0,T;L^q(\Omega))\cap
	C^0([0,T];W^{2-2/p,q}(\Omega)),
$$
and there exists a constant $C>0$ such that
\begin{align*}
  \|u\|_{L^\infty(0,T;W^{2-2/p,q}(\Omega))} 
  &+ \|u\|_{L^p(0,T;W^{2,q}(\Omega))}
  + \|\pa_t u\|_{L^p(0,T;L^q(\Omega))} \\
  &\le C\big(\|f\|_{L^p(0,T;L^q(\Omega))} 
  + \|u^0\|_{W^{2-2/p,q}(\Omega)}\big).
\end{align*}
\end{theorem}

\end{appendix}


\end{document}